\documentclass[letterpaper, 10 pt]{ieeeconf}  % Comment this line out if you need a4paper
\usepackage[letterpaper, left=0.75in, right=0.75in, bottom=0.75in, top=0.75in]{geometry}

\IEEEoverridecommandlockouts                              % This command is only needed if 
\usepackage{graphics} % for pdf, bitmapped graphics files
\usepackage{epsfig} % for postscript graphics files
\usepackage{epstopdf}
\usepackage{amsmath} % assumes amsmath package installed
\usepackage{amssymb}  % assumes amsmath package installed
\usepackage{amsfonts}
\usepackage{subfigure}

\usepackage{xspace}
\usepackage{soul}
\usepackage{mathtools}
\usepackage{makecell}
\usepackage{todonotes}
\usepackage{nicefrac}

\newcommand{\ra}{\rightarrow}
\newcommand{\rra}{\rightrightarrows}

\newcommand{\ie}{\unskip, i.\,e.,\xspace}
\newcommand{\eg}{\unskip, e.\,g.,\xspace}

\newcommand{\sut}{\text{s.\,t.\,}}

\newcommand{\nrm}[1]{\left\lVert#1\right\rVert}

\newcommand{\abs}[1]{\left\lvert#1\right\rvert}

\newcommand{\N}{\ensuremath{\mathbb{N}}}

\newcommand{\Q}{\ensuremath{\mathbb{Q}}}
\newcommand{\R}{\ensuremath{\mathbb{R}}}

\newcommand{\X}{\ensuremath{\mathbb{X}}}

\newcommand{\sm}{\ensuremath{\setminus}}

\newcommand*\diff{\mathop{}\!\mathrm{d}}
\newcommand{\eps}{\ensuremath{\varepsilon}}

\newcommand{\spc}{\ensuremath{\,\,}}

\newcommand{\dom}{\ensuremath{\text{dom}}}

\newcommand{\ball}{\ensuremath{\mathcal B}}

\definecolor{dgreen}{rgb}{0.0, 0.5, 0.0}

\newtheorem{dfn}{Definition}

\newtheorem{thm}{Theorem}
\newtheorem{crl}{Corollary}
\newtheorem{rem}{Remark}

\newcommand{\D}{\ensuremath{\mathcal{D}}}

\newcommand{\B}{\ensuremath{\mathbb{B}}}
\newcommand{\J}{\ensuremath{\mathbb{J}}}

\newcommand{\cube}{\ensuremath{\mathcal{H}}}

\newcommand{\wrt}{w.~r.~t.\xspace}
\newcommand{\wlg}{w.~l.~g.\xspace}

\newcommand{\sea}{\searrow}

\title{Towards a constructive framework for control theory}

\author{Pavel Osinenko% <-this % stops a space
\thanks{
Email: \texttt{p.osinenko@yandex.ru}
}}

\begin{document}

\IEEEaftertitletext{\vspace{-30pt}}

\maketitle

\thispagestyle{empty}
\pagestyle{empty}

%%%%%%%%%%%%%%%%%%%%%%%%%%%%%%%%%%%%%%%%%%%%%%%%%%%%%%%%%%%%%%%%%%%%%%%%%%%%%%%%
\begin{abstract}
This work presents a framework for control theory based on constructive analysis to account for discrepancy between mathematical results and their implementation in a computer, also referred to as computational uncertainty.
In control engineering, the latter is usually either neglected or considered submerged into some other type of uncertainty, such as system noise, and addressed within robust control.
However, even robust control methods may be compromised when the mathematical objects involved in the respective algorithms fail to exist in exact form and subsequently fail to satisfy the required properties.
For instance, in general stabilization using a control Lyapunov function, computational uncertainty may distort stability certificates or even destabilize the system despite robustness of the stabilization routine with regards to system, actuator and measurement noise.
In fact, battling numerical problems in practical implementation of controllers is common among control engineers. 
Such observations indicate that computational uncertainty should indeed be addressed explicitly in controller synthesis and system analysis.
The major contribution here is a fairly general framework for proof techniques in analysis and synthesis of control systems based on constructive analysis which explicitly states that every computation be doable only up to a finite precision thus accounting for computational uncertainty.
A series of previous works is overviewed, including constructive system stability and stabilization, approximate optimal controls, eigenvalue problems, Caratheodory trajectories, measurable selectors.
Additionally, a new constructive version of the Danskin's theorem, which is crucial in adversarial defense, is presented.
\end{abstract}

%\begin{keywords}
%Constructive analysis, computational uncertainty, foundations of control theory
%\end{keywords}
%%%%%%%%%%%%%%%%%%%%%%%%%%%%%%%%%%%%%%%%%%%%%%%%%%%%%%%%%%%%%%%%%%%%%%%%%%%%%%%%
\section{Introduction}\label{sec:intro}

% Comp. uncertainty in control

%\red{Brief survey of software at the end of this section actually}

As stated above, computational uncertainty in control oftentimes poses serious issues and should in general be differentiated from other types of uncertainty \cite{vasile2017optimising}.
It may occur when certain idealized mathematical objects fail to exist in practice, such as exact optimizers.
For instance, Sutherland et al. \cite{sutherland2019closed} recently showed loss of Lyapunov stability under non-uniqueness of optimal controls due to computational uncertainty in model-predictive control. 
A number of approaches in tackling computational uncertainty used computable analysis of Weihrauch \cite{Weihrauch2012-computable-analys}, where each computation is required to terminate.
%For instance, Collins \cite{collins2005continuity,collins2009computable} suggested it for reachable sets and, subsequently, as a general foundation of control theory.
For instance, Collins \cite{collins2009computable} suggested it as a general foundation of control theory.
A similar proposal was made in \cite{buescu2011computability} studying links between dynamical systems and computability.
In the context of planar dynamical systems, computability of basins of attraction was considered in \cite{gracca2011computability}.
Formal methods, such SMT (Satisfiability Modulo Theory) found applications to tackle the issue of computational uncertainty.
Shoukry et al. \cite{shoukry2015sound} used SMT-solvers for state estimation of linear dynamic systems.
Bessa et al. \cite{bessa2016formal} used them for stability verification of uncertain linear systems.
Another noticeable tool is the Coq proof assistant.
Cohen and Rouhling \cite{cohen2017formal} used it, particularly the Coqelicot library, for formalization of the LaSalle's principle.
The axiomatization of reals was classical though, but they believe the results to be close to being constructive. 
The same tool was used in \cite{rouhling2018formal} for formalization of control of inverted pendulum, and in \cite{gallois2018coq} -- for formalization of digital filters.
Jasim and Veres \cite{jasim2017towards} stressed the help of formal methods to assist system analysis, commonly done manually by an engineer. 
%Various formal logical systems found wide applications, perhaps, most notably temporal and differential dynamic logic \cite{belta2017formal,tan2020deductive}.
Various formal logical systems found wide applications, perhaps, most notably temporal and differential dynamic logic \cite{tan2020deductive}.
The latter is realized in the software called KeYmaera X.
Gao et al. \cite{gao2019numerically} developed a framework to argument about stability in terms of $\eps$-stability and $\eps$-Lyapunov functions to address for computational uncertainty.
%Quantification for the latter was suggested using fuzzy type-2 by Alipouri and Poshtan \cite{alipouri2017robust}.
Tsiotras and Mesbahi \cite{Tsiotras2017-alg-ctrl} stressed the issues of computational uncertainty in what they called ``algorithmic control theory''.

\textbf{Summary and contribution:} it is clear that computational uncertainty is being attacked from various directions in the control community with different approaches having their pros and cons.
For instance, despite attractiveness of computable analysis, its ambient logic is classical and although the computations are required to terminate, there is no way to say exactly after how many iterations.
Formal verification software is gaining attention, but it is still computationally heavy and requires special training.
In this work, we suggest another framework, based on constructive analysis, which has the advantage that its style is quite close to the usual business of a control engineer, just done with special care.
A brief description is given in Section \ref{sec:framework} followed by an overview of the results achieved so far, including in the field of optimal control, stabilization, system analysis.
Whereas the detailed proofs can be found in the referenced works, outlines and key steps are provided.
As a new result, an approximate and constructive version of the Danskin's theorem, which is used in adversarial defense and reinforcement learning methods, is presented in Section \ref{sec:Danskin}.

\textbf{Notation and abbreviations.} Convergence of $t$ to $a$ from the right: $t \sea a$.
%A Euclidean space with an $\infty$-metric will be denoted $(\R^n,d_\infty)$ for some $n \in \N$.
A closed ball with a radius $r$ centered at $x$: $\ball_r(x)$, or just $\ball_r$ if $x=0$. 
A closed hypercube with a side length $r$ centered at $x$: $\cube_r(x)$, or just $\cube_r$ if $x=0$.
%Minkowski sum: $\oplus$ (in particular, $\A \oplus r$ is the same as $\{x+y: x \in \A, y \in \ball_r\}$).
Euclidean distance, or $\sup$-norm of a function: $\nrm{\bullet}$.
Domain of a function: $\dom$.
%\textbf{Abbreviations.}
``Such that'': \sut.
``With respect to'': \wrt.
``Without loss of generality'': \wlg.
%``Sample-and-hold'': \SH.

%%%%%%%%%%%%%%%%%%%%%%%%%%%%%%%%%%%%%%%%%%%%%%%%%%%%%%%%%%%%%%%%%%%%%%%%%%%%%%%%
\section{The framework}\label{sec:framework}

%\subsection{Basics}\label{sub:framework-basics}

%The foundation of the framework presented here is the constructive analysis of Bishop \cite{Bishop1985-constr-analysis,Schwichtenberg2015-constr-analys}.
The foundation of the framework presented here is the constructive analysis of Bishop \cite{Bishop1985-constr-analysis}.
In this work, quite a bit of foundations was already tackled, but the field is actively developed -- recent works offer extensive coverage of such subjects as stochastic processes \cite{chan2019foundations} and abstract algebra \cite{Lombardi2014-constr-commut-alg}.
A fresh presentation, close to program code, can be found in \cite{Schwichtenberg2015-constr-analys}.
An interested reader may also take a look at this nice and easy-to-read recent explanation: \cite{havea2020being}.
The essence of constructive analysis is that everything must have a sound and finite computational content.
In this regard, constructive analysis does not ``suppress'' computational uncertainty, but rather takes an explicit account thereof.
For instance, a real-valued vector $x=(x^1,\dots,x^n)^\top$ is treated as an algorithm that computes rational approximations $\{x_i\}_i$ with a convergence certificate like $\forall i, j \spc \max_{k = 1, \dots, n} |x^k_i-x^k_j| \le \nicefrac 1 i + \nicefrac 1 j$.
This is in contrast to the classical definition where no convergence information is required -- a vector is simply a tuple of equivalence classes of Cauchy sequences, not necessarily computable.
In practice, we are always dealing with some $x_i$ depending on the computational device's precision.
Sets are also treated with care in constructive analysis as there are plenty of examples which are computationally problematic \cite{Bridges1996-sets}.
For instance, bounded sets are in general not necessarily \textit{totally bounded} -- to mean enclosing an algorithm that computes finite meshes approximating the said set.
This goes as follows: a set $X \subset \R^n$ is called totally bounded if there is an algorithm that, for any $\eps$, constructs a finite set $\{x_i\}_{i=1}^N$ of distinct points in $X$ such that any $x \in X$ lies within an $\eps$-ball centered at some $x_i$.
If a totally bounded set is complete, then it is called compact (notice, the related finite-mesh approximation algorithm is still encoded in the definition of the compact set!)
The distance-to-set function $\nrm{x - A} \triangleq \inf_{y \in A} \nrm{x - y}$ is also not always finitely computable -- those sets, whose function is, are called \textit{located}.

Another example of encoding computational content is within the definition of a continuous function.
A function $f: \R^n \ra \R^m$ is a pair of algorithms: one computes rational approximations to $f(x)$ from rational approximations to $x \in \R^n$, and the second one, denoted $\omega_f$ and called \textit{modulus of continuity}, satisfies the formula: $\forall \eps>0, c \in \R^n, r>0 , \forall x, y \in \ball_r(c) \spc \nrm{x-y} \le \omega(\eps, c, r) \implies \nrm{f(x) - f(y)} \le \eps$.
Such a modulus of continuity is called to have the $\omega$-format.
We will also use the $\mu$-format to mean $\forall x,y \spc \nrm{f(x) - f(y)} \le \mu_f(\nrm{x-y})$ with $\mu_f: \R \ra \R$ positive-definite.
Constructively, these two formats are not equivalent, unlike classically.
What makes a difference in working constructively is that attention should be paid to the objects or claims without a finite computational content.
So is \eg convergent subsequence extraction (also called sequential compactness argument) which caries no information whatsoever about how to actually do this extraction.
Such arguments are commonly used in \eg optimal control which is discussed in more detail in Section \ref{sec:optimal}.
Consequently, it is not constructively valid to claim existence of exact optimizers in general.
Still, the most evident difference to the classical reasoning is undecidability of $a=b$ vs. $a \ne b$ for arbitrary real numbers $a ,b$.
%A consequence of if is \eg that eigenvectors do not exist in exact form in general.
Despite the said limitations, constructive analysis does offer a powerful apparatus for control engineer as long as a clear correspondence of pure mathematical objects and their computational realizations is concerned. 
This is supported, in particular, by the famous realizability interpretation \cite{troelstra2014constructivism} that states that every constructive existence claim is isomorphic to a finite algorithm.
%Furthermore, the brief overview in the forthcoming sections is a demonstration that application of constructive analysis to control theory definitely has potential.

%Now, the context of the current work is briefly discussed.
%The central point is a general dynamical system
%\begin{equation}
%	\label{eqn:sys}
%	\D x = f(x, u),
%\end{equation}
%where $\D$ is a suitable differential operator.
%In control systems, $\D$ is rarely the ordinary derivative (whence we would just write $\dot x$) due to the possible presence of discontinuities in the right-hand side.
%The latter may come \eg from a sampled character of the control signal realized by a digital device that introduces time discontinuities.
%In this case, $\D$ should at least be taken as the right derivative:
%\begin{equation}
%	\label{eqn:right-der}
%	\D^+ x(t) \triangleq \lim_{\tau \sea 0} \tfrac{x(\tau+t) - x(\tau)}{t}.
%\end{equation}

%%%%%%%%%%%%%%%%%%%%%%%%%%%%%%%%%%%%%%%%%%%%%%%%%%%%%%%%%%%%%%%%%%%%%%%%%%%%%%%%
\section{Optimal control and stabilization} \label{sec:optimal}

We start with optimal control which is undoubtedly the central branch of control theory -- it is worth noting that reinforcement learning, one of the most vanguard methods of control for the time being, is based upon optimal control theory, dynamic programming in particular.
In turn, central to optimal control is the variety of extremum value theorems for function spaces, which are used to show existence of optimal controls altogether, either as functions of time or state.

It is precisely the extremum value theorems (EVTs) that, in practice, suffer from computational uncertainty.
The problem is that the most of the related proofs use a sequential compactness argument of the following kind: one constructs a bounded sequence of controls and then extracts a convergent subsequence from it.
In practice, a naive iterative computation of optimizing controls often fails to converge, particularly due to possible non-uniqueness of optimal controls.
Consequently, this non-uniqueness may pose formidable difficulties -- recall \eg \cite{sutherland2019closed}.
Therefore, constructively, we can only rely on approximately optimal controls in general as per:
%It is possible to constructively prove, using an essentially different to the classical technique, the following theorem, which is fairly sufficient for practical needs:

\begin{thm}[Constructive functional EVT \cite{Osinenko2018-constr-aEVTfnc-Euclid}]
	\label{thm:aEVT-fnc}
	Consider $\mathcal U$, the space of all equi-Lipschitz and equi-bounded functions from a compact set $\X \subset \R^n$ to $\R^m$, and $J$, a uniformly continuous (cost) functional on $\mathcal U$ (``equi'' here to mean having a common Lipschitz constant and a common bound, respectively).
	For any $\eps > 0$, there exists a $\kappa^\eps \in \mathcal U$ such that $J[\kappa^\eps] - \eps \le \inf_{\mathcal U} J$.
\end{thm}

\begin{proof}(\textit{Outline})
	\textbf{Step 1: $\mathcal U$ is totally bounded}.
	To effectively construct an approximate optimizer $\kappa^\eps$, that delivers an approximate optimal value of the cost functional $J$ up to a prescribed precision $\eps$, we need first to ensure that $\mathcal U$ is totally bounded.
	It suffices to show that the subsets $Y : =\left\{ \left(\kappa(x_{1}), \dots, \kappa(x_N)\right): \kappa \in \mathcal U \right\}$ of $\R^N$ with the product metric are totally bounded for any finite set $\X_0=\{x_1, \dots, x_N\}$ of distinct points in $\X$.
	Then, we apply the constructive Arzela–Ascoli’s lemma \cite[p. 100]{Bishop1985-constr-analysis} to conclude that $\mathcal U$ is totally bounded.
	To show $Y$ is totally bounded for a fixed $\X_0$, let $\kappa \in \mathcal U$ be arbitrary.
	Let $K$ be the common bound on the functions in $\mathcal U$ in the sense of: $\forall \kappa \in \mathcal U \spc \nrm{\kappa} \le K$.
	We construct, inductively over $\X_0$, for any prescribed precision $\delta>0$, a piece-wise linear function $\kappa_0$ so that $\kappa_0$ is within $\delta$ to $\kappa$ at $\X_0$ and has the Lipschitz constant $L$ -- the common one for all the functions in $\mathcal U$.
	The idea is to carefully choose a mesh $\mathbb K_0$ on $\ball_K$ -- where $K$ is the common bound on the functions in $\mathcal U$ in the sense of: $\forall \kappa \in \mathcal U \spc \nrm{\kappa} \le K$ -- so as to approximate $\kappa$ by $\kappa_0$ up to the desired precision, while $\kappa_0$ takes values precisely at the nodes of the said mesh.
	After the values of $\kappa_0$ were determined on $\X_0$, we need to extend it to the whole space $\X$.
	To do that, we apply the geometric construction known as the Brehm's extension theorem \cite{Akopyan2008}.
	This theorem applies constructively provided that the points in $\X_0$ and $\mathbb K_0$ possess solely rational coordinates (which may always be assumed) whence all the involved geometric transformations are algebraic \ie they map points with algebraic coordinates to points with algebraic coordinates.
	The trick is to use Lemma 4.1 from \cite[p.~8]{Beeson1980}, which allows to decide whether $x = y$ or $x \ne y$ for arbitrary algebraic numbers $x, y$.
	\textbf{Step 2: approximate optimizer}.
	We construct a desired $\kappa^\eps$ by splitting $\mathcal U$ into a finite set of piece-wise linear functions, picking a minimal one, and claiming the desired property $J[\kappa^\eps] - \eps \le \inf_{\mathcal U} J$ using the continuity modulus of $J$.
\end{proof}

\begin{crl}(Smooth $\eps$-optimizers)
	If all the functions in $\mathcal U$ have equi-bounded derivatives up to order $d$, an approximate optimizer may be found smooth up to order $d$ as well.
\end{crl}

\begin{proof}
	The construction is the same as above, while simply applying a smooth molifier to the resulting piece-wise linear function (see \cite[Appendix]{Osinenko2018-constr-aEVTfnc}).
\end{proof}

\begin{rem}
	The set $\X$ may be just totally bounded, not necessarily compact.
	The theorem trivially applies for product spaces \eg if we augment the domain of control policies to be $\X \times [0, T], T>0$ and assume equi-Lipschitzness and -boundedness of $\mathcal U$ in the second (time) argument.
	In fact, $\kappa \in \mathcal U$ may have jump discontinuities in time so long as they are at fixed points which may be interpreted as time samples (more on sample-and-hold systems down below).
\end{rem}

The justification of Theorem \ref{thm:aEVT-fnc} is that physically every signal is bounded and has a finite rate of change.
For a generic optimal control problem $\min_{\kappa \in \mathcal U} J[\kappa] \spc \sut \D x = f(x, u)$, where $\D$ is a suitable differential operator, we can apply Theorem \ref{thm:aEVT-fnc} if $\mathcal{U}$ is a located subset of a space of all equi-Lipschitz and equi-bounded functions (with the aid of \cite[Lemma~4.3]{Ye2011-SF}).
Otherwise, if the system is \eg input-to-state stable, we have $\forall t \spc \forall \kappa \in \mathcal U \spc \nrm{x(t)} \le \alpha(\nrm{x(0), t}) + \beta(\nrm{\kappa})$ with $\alpha$ of class $\mathcal {KL}$ and $\beta$ -- of $\mathcal K$.
In this case, we may derive a uniform bound on $\nrm{x}$ for all policies in $\mathcal U$, using their common bound, and apply Theorem \ref{thm:aEVT-fnc} directly.
%It should be noted, however, that in presence of state constraints, deriving the class of admissible controls is classically also a general problem.
Further examples of applications of this constructive theorem include dynamic programming and reinforcement learning and may be found in \cite{Osinenko2018-constr-aEVTfnc}.
%If the system is ISS, then we can indeed find a space of admissible control which is a set of all functions from some X with a common bound and Lipschitz constant, because:
%
%||x(t)|| <= alpha(||x0||, t) + beta(||u||)
%
%Then, if we require, for all x, ||kappa(x)|| <= u_bar, we can use ISS to derive a bound x_bar. We have then ||x(t)|| <= x_bar and kappa(x(t)) is well defined. X = ball(0, x_bar), U = ball(0, u_bar) and x(0) must be in some suitable X0.

Now, consider a general problem of stabilization using a control Lyapunov function (CLF).
%Now, consider a general problem of stabilization using a (generally non-smooth) control Lyapunov function (CLF).
First off, existence of a smooth CLF is rarely the case \cite{Sontag1990-stabilization-survey}.
For instance, the most of the computationally obtained ones are non-smooth \cite{Giesl2015-review}.
Starting with a non-smooth CLF, the stabilization can be practical at best -- to mean convergence to any desired small vicinity of the equilibrium.
%Starting with such one, the stabilization can be practical at best -- to mean convergence to any desired small vicinity of the equilibrium.
At the same time, to avoid problems with the existence and uniqueness of system trajectories, a control policy $\kappa$ is usually sampled to get a sample-and-hold system $\D x = f(x, \kappa^\eta(x)), \kappa^\eta(x(t)) :\equiv \kappa(x(k \eta)), t \in [k \eta, (k+1) \eta]$.
Almost all stabilization techniques in this case use an optimization at each sampling time step to compute $\kappa^\eta$.
So are \eg Dini aiming, steepest descent feedback and optimization-based feedback, inf-convolution feedback \cite{Braun2017-SH-stabilization-Dini-aim}.
%\cite{Kellett2004-Dini-aim,Braun2017-SH-stabilization-Dini-aim,Clarke1997-stabilization}
To show practical stabilization, the major focus is to determine an upper bound on the sampling time $\eta$ \cite{Clarke2011-discont-stabilization}.
%\cite{Clarke1997-stabilization,Clarke2009-slid-mode-stab,Clarke2011-discont-stabilization}.
%When it comes to robustness, system, actuator and measurement noise were addressed \cite{Ledyaev1997-stabilization-meas-err,Sontag1999-stabilization-disturb}.
When it comes to robustness, system, actuator and measurement noise were addressed \cite{Sontag1999-stabilization-disturb}.
However, the involved optimization was always assumed exact and so computational uncertainty was neglected, which might pose problems.
%Recently, practical stabilization was shown under approximate optimizers \cite{Osinenko2018-pract-stabilization,Osinenko2019-rob-pract-stab} (see Section \ref{sec:cases} for a discussion on related case studies).
Recently, practical stabilization was shown under approximate optimizers \cite{Osinenko2018-pract-stabilization} (see Section \ref{sec:cases} for a discussion on related case studies).
It should be noted here that explicit account for inexact optimization turned out not to be as trivial, as one might have suspected.
%Recently, it was shown that computational uncertainty could not actually be submerged into another one -- system, actuator or measurement -- in general and could have a great impact on stabilization quality \cite{Osinenko2018-pract-stabilization,Osinenko2019-rob-pract-stab}.
%Explicit account of approximate optimizers turned out not to be as trivial, as one might have suspected, which again supports the idea that computational uncertainty is worth explicit account in control.
%Inexact optimization is not the only problem in proving practical stabilization under computational uncertainty.
%Effective computation of sampling time bounds is also not in general possible and requires posing certain conditions on the CLF \cite{Osinenko2018-constr-SMC,Osinenko2019-rob-pract-stab}.
%In a case study of sliding-mode traction control \cite{Osinenko2018-constr-SMC}, a practically satisfactory bound of 1 ms was achieved under a proposed effective computation using constructive analysis.
%We conclude this paragraph by stating that verified stabilization proofs require constructive approach to yield realistic results. 

\section{Danskin's theorem} \label{sec:Danskin}

%In this section, we constructively study the famous Danskin's theorem, which is foundational in adversarial robustness \cite{maini2020adversarial} and is used in certain reinforcement learning methods \cite{tan2020robustifying,kolaric2020local}.
In this section, we constructively study the famous Danskin's theorem, which is foundational in adversarial robustness \cite{maini2020adversarial} and is used in certain reinforcement learning methods \cite{kolaric2020local}.
The Danskin's theorem is closely related to the extremum value theorems and its classical proofs heavily use sequential compactness arguments.
Here, we prove its constructive version using approximate optimizers.
This poses the major difference to the classical theorem whose statement is based on exact optimizers.
The idea of the proof is to work directly with the sets of approximate optimizers instead of resorting to sequential compactness.

We will use the following \textit{directional super-derivative} (of a function $\psi: \R^n \ra \R$ in direction of a vector $v$) will be used:
%\begin{equation}
%	\label{eqn:dir-super}
	$\D^+_v \psi(x) \triangleq \limsup_{\eps \sea 0} \tfrac{\psi(x + \eps v) - \psi(x)}{\eps}$.
%\end{equation}
By analogy, $\liminf$ in the above will be used for the \textit{directional sub-derivative} $\D^-_v \psi(x)$.
If both coincide, the common limit is simply the directional derivative $\D_v \psi(x)$.
\begin{thm}[Constructive Danskin's theorem]
	\label{thm:Danskin}
	Consider a continuously differentiable function $\varphi: \X \times \Theta \ra \R$ with $\X \subseteq \R^n$, compact $\Theta \subset \R^p$.
	Let $\psi: \X \ra \R$ be defined by $\psi(x)=\max_{\theta \in \Theta} \varphi(x, \theta)$.
	Suppose $E^\delta_\Theta(x) := \{ \theta \in \Theta: \abs{\varphi(x, \theta) - \psi(x)} \le \delta \}$ -- the sets of $\delta$-optimizers of $\varphi$ at $x$ -- are totally bounded for any $x, \delta >0$.	
	Then, $\psi$ is continuous with the same modulus as $\varphi$ \wrt $x$ and the directional derivative of it satisfies:
	\begin{equation}
		\label{eqn:Danskin}
		\begin{array}{l}
			\forall x \spc \forall v \in \R^n \spc \forall \delta > 0 \spc \D_v \psi(x) = \max_{\theta \in E^\delta_\Theta(x)} \D_v \varphi(x, \theta), \\
			\forall \theta^\delta(x) \in E^\delta_\Theta(x) \spc \abs{ \D_v \psi(x) - \D_v \varphi(x, \theta^\delta(x)) } \le \delta.
		\end{array}	 
	\end{equation} 
\end{thm}

\begin{proof}
	For the continuity part, let $\mu^x_\varphi$ be the continuity modulus of $\varphi$ \wrt $x$ to mean: $\forall \theta \in \Theta \spc \varphi(x,\theta) - \varphi(y,\theta) \le \mu^x_\varphi(\nrm{x-y})$.
	In particular, the latter holds with $\theta^\delta(x)$, a $\delta$-optimizer for an arbitrary $\delta>0$, in place of $\theta$.
	Observe that $\varphi(x,\theta^\delta(x)) - \varphi(y,\theta^\delta(x)) \ge \varphi(x,\theta^\delta(x)) - \max_{\theta \in \Theta} \varphi(y, \theta)$.
	So, $\varphi(x,\theta^\delta(x)) - \max_{\theta \in \Theta} \varphi(y, \theta) \le  \mu^x_\varphi(\nrm{x-y})$.
	Now, using a $\delta$-optimizer at $y$, we have: $\varphi(x,\theta^\delta(x)) - \varphi(y,\theta^\delta(y)) \le \mu^x_\varphi(\nrm{x-y}) + \delta$.
	Thus, $\psi(x) - \psi(y) \le \mu^x_\varphi(\nrm{x-y}) + \delta + 2 \delta$, where the last two $\delta$s relate $\psi(x), \psi(y)$ to their respective $\delta$-maximal values.
	Reversing the order of $x, y$ and observing that $\delta$ was arbitrary, it follows that $\abs{\psi(x) - \psi(y)} \le \mu^x_\varphi(\nrm{x-y})$ as required.
	Fix an $x, v, \delta>0$ and observe the following:
	\begin{equation}
	\label{eqn:Danskin-delta-over-2}
		\tfrac{\psi(x+\eps v) - \psi(x)}{\eps} \le \tfrac{\varphi(x + \eps v, \theta^{\eps \nicefrac \delta 2}(x + \eps v)) - \psi(x)}{\eps} + \tfrac \delta 2
	\end{equation}
	for $\theta^{\eps \nicefrac \delta 2}(x + \eps v) \in E^{\eps \nicefrac \delta 2}_\Theta(x + \eps v)$.	
	Assume \wlg that $\eps \le 1$ whence $\forall x \spc \forall \delta >0 \spc E^{\delta \eps}_\Theta(x) \subseteq E^{\delta}_\Theta(x)$ and, therefore,
	\begin{equation}
	\label{eqn:Danskin-max-quotient}
		\tfrac{\varphi(x + \eps v, \theta^{\eps \nicefrac \delta 2}(x + \eps v)) - \psi(x)}{\eps} \le \max_{\theta \in E^{\nicefrac \delta 2}_\Theta(x + \eps v)} \tfrac{\varphi(x + \eps v, \theta) - \psi(x)}{\eps}
	\end{equation}
	by the definition of maximum.
	Since $\varphi$ is continuous, we can always pick $\eps$ small enough (possibly depending on $x, v$) that $E^{\nicefrac \delta 2}_\Theta(x + \eps v) \subseteq E^{\delta}_\Theta(x)$ and so:
	\begin{equation}
	\label{eqn:Danskin-maxima}
		\max_{\theta \in E^{\nicefrac \delta 2}_\Theta(x + \eps v)} \tfrac{\varphi(x + \eps v, \theta) - \psi(x)}{\eps} \le \max_{\theta \in E^{\delta}_\Theta(x)} \tfrac{\varphi(x + \eps v, \theta) - \psi(x)}{\eps}.
	\end{equation}		
	Therefore, combining \eqref{eqn:Danskin-delta-over-2}, \eqref{eqn:Danskin-max-quotient} and \eqref{eqn:Danskin-maxima}, we obtain:
	\begin{equation}
	\label{eqn:Danskin-quotient-max-delta}
		\tfrac{\psi(x+\eps v) - \psi(x)}{\eps} \le \max_{\theta \in E^{\delta}_\Theta(x)} \tfrac{\varphi(x + \eps v, \theta) - \psi(x)}{\eps} + \tfrac \delta 2.
	\end{equation}
	Now, observe that, in general, for any $\delta>0$, $\psi(x) \le \max_{\theta \in E^{\delta}_\Theta(x)} \varphi(x, \theta) + \delta$ and, since $\delta$ is arbitrary, $\psi(x) \le \max_{\theta \in E^{\delta}_\Theta(x)} \varphi(x, \theta)$. 	
	By the same token, \eqref{eqn:Danskin-quotient-max-delta} actually reduces to
	\begin{equation}
	\label{eqn:Danskin-quotient-max}
		\tfrac{\psi(x+\eps v) - \psi(x)}{\eps} \le \max_{\theta \in E^{\delta}_\Theta(x)} \tfrac{\varphi(x + \eps v, \theta) - \psi(x)}{\eps}.
	\end{equation}	
	Observe that $\forall x,\theta -\psi(x) \le -\varphi(x, \theta)$, so:
	\begin{equation}
		\label{eqn:Danskin-maxmax}
		\max_{\theta \in E^{\delta}_\Theta(x)} \tfrac{\varphi(x + \eps v, \theta) - \psi(x)}{\eps} \le \max_{\theta \in E^{\delta}_\Theta(x)} \tfrac{\varphi(x + \eps v, \theta) - \varphi(x, \theta)}{\eps}.
	\end{equation}
	Since $\varphi$ is continuously differentiable, for any $\eps_1>0$ there is an $\bar \eps > 0$ \sut
	\begin{equation}
		\label{eqn:Danskin-eps1}
		\forall \eps \le \bar \eps \spc \forall \theta \in \Theta \spc \tfrac{\varphi(x + \eps v, \theta) - \varphi(x, \theta)}{\eps} \le \D_v \varphi(x, \theta) + \eps_1.
	\end{equation}
	Applying maximum on both sides yields, for $\eps \le \bar \eps$:
	\begin{equation}
		\label{eqn:Danskin-max-der}
		\max_{\theta \in E^{\delta}_\Theta(x)} \tfrac{\varphi(x + \eps v, \theta) - \varphi(x, \theta)}{\eps} \le \max_{\theta \in E^{\delta}_\Theta(x)} \D_v \varphi(x, \theta) + \eps_1
	\end{equation}
	Combining, \eqref{eqn:Danskin-quotient-max}, \eqref{eqn:Danskin-maxmax}, \eqref{eqn:Danskin-eps1}, \eqref{eqn:Danskin-max-der} yields, for $\eps \le \bar \eps$:
	\begin{equation}
		\label{eqn:Danskin-psi-der}
		\tfrac{\psi(x+\eps v) - \psi(x)}{\eps} \le \max_{\theta \in E^\delta_\Theta(x)} \D_v \varphi(x, \theta) + \eps_1.
	\end{equation}	
	Applying $\limsup_{\eps \sea 0}$ on both sides (which acts trivially on the right-hand side) and observing that $\eps_1$ was arbitrary, conclude that
	\begin{equation}
		\label{eqn:Danskin-final-above}
		\D^+_v \psi(x) \le \max_{\theta \in E^\delta_\Theta(x)} \D_v \varphi(x, \theta).
	\end{equation}
	For the other direction, observe that
	\[
		\forall \theta \in \Theta \spc \tfrac{\psi(x+\eps v) - \psi(x)}{\eps} \ge \tfrac{\varphi(x + \eps v, \theta) - \varphi(x, \theta^{\eps \delta}(x))}{\eps} - \delta.
	\]
	So, in particular,
	\[
		\tfrac{\psi(x+\eps v) - \psi(x)}{\eps} \ge \tfrac{\varphi(x + \eps v, \theta^{\delta \eps}(x)) - \varphi(x, \theta^{\eps \delta}(x))}{\eps} - \delta.
	\]
	But since $\theta^{\eps \delta}(x)$ was arbitrary from $E^{\delta}_\Theta(x)$ (provided that $\eps < 1$) we may write $\max_{\theta \in E^{\delta}_\Theta(x)}$ in front of the quotient in the right-hand side of the above.
	The rest of the argument is the same as before, but applying $\liminf$ instead of $\limsup$ and noticing that $\delta$ was arbitrary, yielding
%	\begin{equation}
%		\label{eqn:Danskin-final-below}
		$\D^-_v \psi(x) \ge \max_{\theta \in E^\delta_\Theta(x)} \D_v \varphi(x, \theta)$.
%	\end{equation}
%	Combining \eqref{eqn:Danskin-final-above} and \eqref{eqn:Danskin-final-below} gives the result.
	Combining this with \eqref{eqn:Danskin-final-above} gives the result.
\end{proof}

%\begin{rem}
%	In general, $\delta$ depends on $x$ and $v$, but the result may be semi-global in the sense that there is such a $\delta$ for any compact $\tilde \X \subset \X, \mathbb V \subset \R^m$, which $x$ and, respectively, $v$ are taken from.
%\end{rem}

%\begin{rem}
%	The continuity part used the $\mu$-format for brevity, but the same result can be shown analogously for the $\omega$-format.
%\end{rem}

%\begin{rem}
%	Notice $\varphi$ does not have to be convex.
%\end{rem}

%%%%%%%%%%%%%%%%%%%%%%%%%%%%%%%%%%%%%%%%%%%%%%%%%%%%%%%%%%%%%%%%%%%%%%%%%%%%%%%%

\section{Selector theorems}\label{sec:selector}

Selector theorems refer to extraction of ordinary functions (called \textit{selectors}) out of set-valued functions and are ubiquitous in control engineering, especially in constructing system trajectories in cases when the right-hand side of the system description is time- or state-discontinuous.
In particular, Filippov solutions, which are often standard to describe system trajectories in such cases \eg in sliding mode systems \cite{levant2016discretization}, are constructed essentially using measurable selectors \cite{Aubin2012-diff-incl}.
%\cite{kachroo1999existence,Slotine1991-nonlin-ctrl,Perruquetti2002-slid-mode,azhmyakov2010optimal,levant2016discretization}
For a dynamical system described by a differential inclusion $\D x \in F(t, x(t)), x(0)=x_0$, where $F$ is a set-valued map \eg the Filippov map, trajectories can be constructed, under certain conditions on $F$, via iterations of the kind $x_{i+1}(t) = x_0 + \int_a^t v_i(\tau) \diff \tau$ with $v$s being measurable selectors extracted from $F(\bullet, x_i( \bullet ))$.
Selectors are also used in optimal control problems, including dynamic programming, viability theory, robust stabilization and related fields.
Aubin \cite{Aubin2011-viab-thr} stressed that the selector theorems were not constructive and so there is no actual algorithm to compute selectors. 
It turns out that under certain conditions on the respective set-valued functions, continuous selectors can actually be found constructively  \cite[Chapter~4]{Waaldijk1996-intuit-top}.

Recently, we showed that extraction of measurable selectors could also be made constructive \cite{Osinenko2018-constr-sel-thm}.
Whereas the full details can be found in the related work, let us outline the result in this section.
Let a \textit{block} be a not necessarily non-empty (closed) hyperrectangle with rational vertices in $\R^n$.
%First, fix a compact $\X \in \R^n$ and let a \textit{block} be a not necessarily non-empty (closed) hyperrectangle with rational vertices in $\R^n$.
Let unions of blocks be called \textit{generalized blocks}.
%Working with generalized blocks $\B = \bigcup_i \B_i$, we will always assume the underlying sequence $\{\B_i\}_i$ be available and so will make use of the notation $\B = \{\B_i\}_i$ for brevity.
% TYPOE HERE. NOT \B, BUT \X IS CORRECT
%For a generalized block $\B = \bigcup_i \B_i$, if any block $A \in \B$ is inside finitely many $\B_i$s, $\B$ is called locally finitely enumerable (or just finitely enumerable, if $\B$ is finite as a sequence).
For a generalized block $\B = \bigcup_i \B_i$, if any block $A \in \R^n$ intersects with only finitely many $\B_i$s, $\B$ is called \textit{locally finitely enumerable} (or just finitely enumerable, if $\B$ is finite as a sequence).
Notice, when dealing with $\B = \bigcup_i \B_i$, we always assume the underlying sequence $\{\B_i\}_i$ be available.
If $\B$ is locally finitely enumerable and $\{\B_i\}_i$ are disjoint, the generalized block is called proper.
Let $\B = \bigcup_i \B_i$ be a locally finitely enumerable generalized block, and $A$ be a block.
Define a map $\mu_A(\B) \triangleq \sum_i \mu(\B_i \cap A )$ where $\mu(\B_i \cap A)$ is the volume of the respective hyperrectangle (possibly empty) resulting from the intersection.
Notice the map $\mu$ generalizes the definition \cite[Chapter~6]{Ye2011-SF}, and thus is not treated as the classical Lebesgue measure.

Fix a compact set $\X \subset \R^n$.
Hence, there exists a block $A$ which contains $\X$.
If a generalized block $\B = \{\B_i\}_i$ is such that each $B_i$ is inside $\X$, we simply write $\mu(\B)$ meaning a suitable ambient bounding block $A$.

We say a sequence of generalized blocks $\mathcal E = \{\B^j\}_j$ is a \textit{representable domain} in $\X$ if for any $\eps > 0$, there exists another generalized block $\J \in \X$ with $\mu(\J) \le \eps$ \sut $\partial_{\mathcal E} \Subset \J \sm \partial_{\mathcal \J}$, and $\X \sm \J$ exists and satisfies $\X \sm \J \subseteq {\mathcal E}$, where $\partial_{\mathcal E}$ is the generalized block which is the union of the boundaries of all $\B^j_i$s and $A \Subset B$ means (constructively) well-contained \ie $\exists \lambda > 0 \spc A + \lambda \subseteq B$.

We will consider in the following measurable (single- and set-valued) functions whose domains are proper, representable, and whose values are located.

The idea behind representable domains is that they entail an algorithm which yields ``arbitrarily thin'' generalized blocks in a way that these domains cover the total space minus the said generalized blocks.
Noticing $F(x)$ are located, let us introduce the following definition.

\begin{dfn}[Representable inverse]
	\label{dfn:represent-inverse}
	A set-valued function $F: \X \rra \R$ with a domain $\cup_i \B_i$ is said to have representable inverse if for any finite sequence $\{r_i\}_i^N$ and $r>0$,	$\cup_{i \le N} \{ x \in \X : \nrm{r_i - F(x)} \le r \}$ is representable.
\end{dfn}
\begin{dfn}[Simple set-valued function]
\label{dfn:simple-SVF}	
	A set-valued function $F: \X \rra \R$ with a domain $\cup_i \B_i$ whose values on each $\B_i$ are finitely enumerable generalized blocks is called simple.
\end{dfn}
% Clearly, the inverse images are weakly finitely adjacent in simple SVFs (see the selector paper).
\begin{dfn}[Regular set-valued function]
\label{dfn:regular-SVF}	
	A set-valued function $F: \X \rra \R$ with a domain $\cup_i \B_i$ defined as
	\[
		\forall i \spc \forall x \in \B_i \spc F(x) = \cup_{k=1}^{N_i} \{ y: \alpha^i_k(x) \le y \le \beta^i_k(x) \}, N_i \in \N
	\]
	with continuous $\alpha^i_k(x), \beta^i_k(x)$ is called regular.
\end{dfn}
A regular set-valued function is a one whose image on each separate $\B_i$ is a finite set of ``chunks'' with boundaries in the form of continuous functions (such a description is fairly general).
Finally, we will need the so called countable reduction, which converts a sequence of generalized blocks in a sequence of proper ones (cf. \cite[Lemma~2]{Osinenko2018-constr-sel-thm}).
With the introduced machinery at hand, we can state the following constructive approximate measurable selector theorem:
\begin{thm}[Constructive selector extraction \cite{Osinenko2018-constr-sel-thm}]
	\label{thm:meas-sel}
	Let $F: \X \rra \R$ be a regular set-valued function.
	Then, for any $\eps>0$ there exists a measurable function $f: \X \ra \R$ \sut $\nrm{F(x) - f(x)} \le \eps$ on a representable domain. 
\end{thm}
\begin{proof} \textit{(Outline)}
	We may assume \wlg that $F$ maps to a unit interval.
	Observe that we can always approximate $F$ by a simple set-valued function $\hat F$ on a representable domain.
	From now, let us fix $\hat F$ to be such an approximation up to the accuracy $\frac \eps 2$.
	The essence of the proof is the following algorithm, starting with $f_1 :\equiv 0$:
	\begin{enumerate}
%		\item set $f_1 :\equiv 0$;
		\item for $k \in \{2, \dots, N\}, \nicefrac{1}{2N} \le \nicefrac \eps 2$, generate $\{r^k_i\}_{i \le N_k}$, a $\frac{1}{2^{k+1}}$-mesh on $[0,1]$;
		\item compute the sets $C^k_i, D^k_i, A^k_i$ as follows:
		\begin{equation*}
			\label{eqn:meas-sel-sets}
			\begin{array}{ll}
				C^k_i & := \left\{ x \in \dom(\hat F) : | r^k_i - \hat F(x) | < \tfrac{1}{2^k} \right\}, \\
				D^k_i & := \left\{ x \in \dom(f_{k-1}) : | r^k_i - f_{k-1}(x) | < \tfrac{1}{2^{k-1}} \right\}, \\
				A^k_i & := C^k_i \cap D^k_i.
			\end{array}
		\end{equation*}
		\item compute $\{Q^k_i\}_{i \le N_k}$ by countable reductions of $\{A^k_i\}_{i \le N_k}$;
		\item set $\dom(f_k) : = \cup_{i \le N_k} Q^k_i$ and $f :\equiv r^k_i$ on $Q^k_i$.
	\end{enumerate}
	Notice that $C^k_i$s always exist and are proper since $\hat F$ is simple.
	The rest of the proof is the same as in \cite[Theorem~2]{Osinenko2018-constr-sel-thm}.
	In brief, we show inductively that $f_k$s are indeed measurable and approximate $\hat F$ as desired.
	We then take the last one, $f:=f_N$ and conclude that $\nrm{F(x) - f(x)} \le \nrm{\hat F(x) - f(x)} + \nrm{\hat F(x) - F(x)} \le \eps$ on $\dom(f)$ as required, noticing that $F(x)$s are located.
\end{proof}

\begin{crl}
	If $F$ has representable inverse, $f_k$ converge to a measurable function with $f(x) \in F(x)$ on a representable domain.
\end{crl}

%%%%%%%%%%%%%%%%%%%%%%%%%%%%%%%%%%%%%%%%%%%%%%%%%%%%%%%%%%%%%%%%%%%%%%%%%%%%%%%%
\section{System analysis}\label{sec:sys-analys}

In this section, we briefly overview selected aspects of system analysis done constructively. 
First, regarding linear systems, as mentioned in Section \ref{sec:framework}, the major difficulty has to do with exact eigenvectors, which consequently complicates various matrix decompositions ubiquitous in classical analyses.
However, with the help of \cite[Lemma~4.1]{Beeson1980}, we can find approximate eigenvectors as per:
\begin{thm}[Constructive eigen-decomposition \cite{Osinenko2018-constr-eig-vec}]
	\label{thm:constr-eig}
	Let $A$ be a complex-valued $n \times n$ matrix with the characteristic polynomial $\chi_A(\lambda)$.
	For any $\eps > 0$, there exist a $k \le n$ linearly independent vectors $\hat v_1, \dots \hat v_k$ and complex numbers $\hat \lambda_1, \dots \hat \lambda_m$ \sut	$\forall i = 1, \dots k  \spc \| A \hat v_i - \hat \lambda_i \hat v_i \| \le \eps$.
\end{thm}
A combination of Theorem \ref{thm:constr-eig} and certain perturbation bounds on matrix exponentials
%(see \eg \cite{lee1993sharp})
\cite{kaagstrom1977bounds,van1977sensitivity,lee1993sharp}
then enables the eigenvalue criterion for stability.
Now, we briefly tackle nonlinear systems and start with trajectory existence.
To this end, consider the following (cf. Definition \ref{dfn:regular-SVF}):
\begin{dfn}[Regular measurable function]
\label{dfn:regular-fnc}	
	A function $f: \X \times \R \ra \R^n, \X \subseteq \R^n$ with a domain $\X \times \cup_i \B_i$ defined as $\forall x \in \X \spc \forall i \spc \forall t \in \B_i \spc f(x, t) = \alpha^i(x, t), i \in \N$
%	\[
%		\forall x \in \X \spc \forall i \spc \forall t \in \B_i \spc f(x, t) = \alpha^i(x, t), i \in \N
%	\]
	with continuous $\alpha^i: \R^n \times \R \ra \R$ is called regular.
\end{dfn}
Respectively, let us call $f$ Lipschitz-regular in $x$ if $\alpha^i$ are Lipschitz in $x$.
We have:
\begin{thm}
\label{thm:Caratheodory-constr}
Consider the initial value problem $D x = f(x, t), x(0) = x_0$ on a hyperrectangle $\X \times [0,T]$ with $f$ being Lipschitz-regular.
There exists a local unique solution in the extended sense \ie $\D x$ satisfies the respective differential equation on a representable domain.
Moreover, the solution depends on the initial condition uniformly continuously.
\end{thm}
The proof of Theorem \ref{thm:Caratheodory-constr} essentially utilized the regularity of $f$ to do the Picard iteration constructively.
Regarding Lyapunov stability, the comparison principles of \cite[Theorem 3.8]{Khalil1996-nonlin-sys} require certain modifications.
First, call a function $w: \R^n \ra \R$ \textit{strictly increasing (in norm)} if there is a map $\nu: \Q^n \times \Q^n \rightarrow \Q_{>0} \spc \sut \forall x, y \in \Q^n \spc \nrm{x} < \nrm{y} \implies w(y) - w(x) > \nu(x, y)$.
With this at hand, we have:
\begin{thm}[\cite{Osinenko2017-constr-stab}]
	\label{thm:constr-stab}
	Let $\X \subset \R^n$ be compact and $\dot x = f(x,t), x \in \X$ be a dynamical system with the equilibrium point $x_e = 0$ and $f$ Lipschitz-continuous in $x$.
	Suppose there exist positive-definite functions $V, w_1, w_2, w_3: \X \times \R_{\ge 0} \ra \R$, $V$ continuously differentiable, $w_1, w_2, w_3$ strictly increasing, $w_3$ Lipschitz continuous, with the following properties:
%	a continuously differentiable 
	\begin{enumerate}
		\item $\forall t \ge 0 \spc \forall x \in \X \spc w_1(x) \le V(x, t) \le w_2(x)$ and there is $\xi > 0$ \sut $\forall \|x\| \ge \|y\| \spc w_2(x) - w_2(y) \ge \xi ( \|x\| - \|y\| )$;     
%		\item there exists a strictly increasing function $w_2: \X \ra \R, w_2(0) = 0$, a positive rational number $\xi$ \sut $\forall \|x\| \ge \|y\| \spc w_2(x) - w_2(y) \ge \xi ( \|x\| - \|y\| )$ and $\forall t \ge 0 \spc \forall x \in \X \spc V(x, t) \le w_2(x)$;
		\item $\forall t \ge 0 \spc \forall x \in \X \spc \dot{V}(x,t) \le - w_3(x)$.
	\end{enumerate}
	Then, $x_e = 0$ is asymptotically stable for any $x_0$ in a set $\X_0 \subseteq \X$ that depends only on the data $f, V, w_1, w_2, w_3$.   
\end{thm}

%%%%%%%%%%%%%%%%%%%%%%%%%%%%%%%%%%%%%%%%%%%%%%%%%%%%%%%%%%%%%%%%%%%%%%%%%%%%%%%%%
%\section{Linear theory}\label{sec:linear}
%
%\red{Possibly join with the next one and call system analysis or smth}
%
%Constr eig val problems (ours).
%My notes on eig val criterion etc.
 
%%%%%%%%%%%%%%%%%%%%%%%%%%%%%%%%%%%%%%%%%%%%%%%%%%%%%%%%%%%%%%%%%%%%%%%%%%%%%%%%%
%\section{Stability theory}\label{sec:stability}
%
%Squeezed version of our IFAC article.
%
%Somewhere also: Caratheodory.

% +RL: Platzer: shield and see the talk on safe RL

%%%%%%%%%%%%%%%%%%%%%%%%%%%%%%%%%%%%%%%%%%%%%%%%%%%%%%%%%%%%%%%%%%%%%%%%%%%%%%%%
\section{Overview of case studies}\label{sec:cases}

Application of constructive analysis to control theory was demonstrated in several case studies.
First, regarding general stabilization, it was revealed, supporting the claims in the end of Section \ref{sec:optimal}, that computational uncertainty has a great impact on stabilization quality and it cannot in general simply be submerged into actuator, system or measurement uncertainty, as was shown in case studies with mobile robots \cite{Osinenko2018-pract-stabilization}.
%, which again supports the idea that computational uncertainty is worth explicit account in control.
%Inexact optimization is not the only problem in proving practical stabilization under computational uncertainty.
Remarkably, effective computation of sampling time in practical sample-and-hold stabilization was enabled by constructive analysis under certain conditions on the involved CLF.
%Effective computation of sampling time bounds is also not in general possible and requires posing certain conditions on the CLF .
So, for instance, in a case study of sliding-mode traction control \cite{Osinenko2018-constr-SMC}, a practically satisfactory bound of 1 ms was achieved under the proposed effective computation using constructive analysis.
Selector Theorem \ref{thm:meas-sel} was used in non-smooth backstepping for a mobile robot parking while reducing chattering compared to a baseline algorithm \cite{Osinenko2018-constr-sel-thm}.
The work \cite{Osinenko2018-constr-eig-vec} applied constructive approximate eigenvectors in an LQR, also demonstrating high influence of computational uncertainty.
Theorem \ref{thm:constr-stab} was used to compute stability certificates via algorithms extracted from the proof in \cite{Osinenko2017-constr-stab}.
%A computational study of parking of a mobile robot by means of non-smooth backstepping showed significant chattering reduction when using extracted selector as per Theorem \ref{thm:meas-sel} \cite{Osinenko2018-constr-sel-thm}. 

\vspace{-10pt}

%%%%%%%%%%%%%%%%%%%%%%%%%%%%%%%%%%%%%%%%%%%%%%%%%%%%%%%%%%%%%%%%%%%%%%%%%%%%%%%%
\section{Conclusion}\label{sec:conclusion}

%Several concluding remarks can be made here.
%First, 
Constructive analysis can be considered a suitable framework for control theory to perform mathematical analyses with explicit account for computational uncertainty.
Not only does it enable the said analyses, it can also reveal computational weaknesses in classical results and enable new computational algorithms for control.
This paper demonstrated a set of constructive results in control theory which supports potentials of the framework.
As for the indicators of when the framework may be resorted to, we may list such arguments as sequential compactness, exact optimizers and existential proofs by contradiction (since they do not construct algorithms to find the related objects).
Numerical troubles with control algorithms may in turn serve as practical indicators for looking at the problem from the constructive standpoint.
\bibliography{bib/adversarial,bib/analysis,bib/automated-reasoning,bib/computable,bib/constr-math,bib/diff-games,bib/discont-DE,bib/Filippov-sol,bib/formal-ctrl,bib/lin-ctrl,bib/logic,bib/model-reduction,bib/MPC,bib/nonlin-ctrl,bib/nonsmooth-analysis,bib/opt-ctrl,bib/Osinenko,bib/perturb-thr,bib/selector-misc,bib/semiconcave,bib/set-thr,bib/sliding-mode,bib/soft,bib/stabilization,bib/stochastic,bib/topology,bib/uncertainty,bib/verified-integration,bib/viability,bib/metric-spaces,bib/DP,bib/constructing-LFs}
\bibliographystyle{IEEEtran}

\end{document}